\documentclass[12pt, reqno]{amsproc}%
\usepackage{amsfonts}
\usepackage{amsmath}
\usepackage{amssymb}
\usepackage{graphicx}%
\providecommand{\U}[1]{\protect\rule{.1in}{.1in}}

\theoremstyle{plain}

\newtheorem{thm}{Theorem}[section]
\newtheorem{lemma}{Lemma}[section]

\newtheorem{cor}{Corollary}[section]

\numberwithin{equation}{section}

\usepackage[notcite,notref, notshowkeys]{ }
\usepackage{amsthm,amsmath,amsfonts}
\usepackage{color,latexsym,amsfonts,amssymb}

\makeatletter
\@addtoreset{equation}{section}

\newcommand{\beq}[1]{\begin{equation} \label{#1}}
\newcommand{\eeq}{\end{equation}}
\newcommand{\bed}{\begin{displaymath}}
\newcommand{\eed}{\end{displaymath}}
\newcommand{\bea}{$$\begin{array}{ll}}
\newcommand{\eea}{\end{array}$$}

\newcommand{\barray}{\begin{array}{ll}}
\newcommand{\earray}{\end{array}}

\renewcommand{\(}{\Big(}
\renewcommand{\)}{\Big)}

\renewcommand{\|}{\Big|}

\newcommand{\p}{\partial}

\newcommand{\bna}{\begin{eqnarray*}}
\newcommand{\ena}{\end{eqnarray*}}

\begin{document}
  \title[the Chang-Marshall inequality for Sobolev functions]{The Chang-Marshall Trace Inequality for Sobolev functions in domains in higher dimensional space $\mathbb{R}^n$}
  \author{Jungang Li}
  \address{Department of Mathematics \\
  Brown University \\
  Providence, RI 02912, USA}
  \email{jungang\_li@brown.edu}
  \author{Guozhen Lu}
  \address{Department of Mathematics \\
  University of Connecticut \\
  Storrs, CT 06269, USA}
  \email{guozhen.lu@uconn.edu}

\begin{abstract}
The main purpose is to address a question proposed to us by S. Y. Alice Chang who asked whether the Chang-Marshall inequality for holomorphic functions on unit disk on the complex plane holds for Sobolev functions on general domains in higher dimensional Euclidean space $\mathbb{R}^n$ for all $n\ge 2$. We answer her question affirmatively.

\end{abstract}

  \thanks{Research is partly supported by a Simons collaboration grant from the Simons Foundation.}

  \date{}
  \dedicatory{ }
  \maketitle

\section{Introduction}

In $\mathbb{R}^n$, it is well known that for any $1 \leq p < n$, there exists a constant $C$ such that the following Sobolev inequality holds
\begin{equation*}\label{Sobolev Inequality}
  ||u||_{L^{p^*}(\mathbb{R}^n)} \leq C ||\nabla u||_{L^p(\mathbb{R}^n)},
\end{equation*}
where $u \in C_0^1(\mathbb{R}^n)$ and $p^*$ denotes the Sobolev conjugation of $p$, i.e. $\frac{1}{p^*} = \frac{1}{p} - \frac{1}{n}$. By this Sobolev inequality, one can conclude that for any bounded domain $\Omega \subset \mathbb{R}^n$, the Sobolev space $W^{1,p}_0(\Omega)$ can be embedded into $L^{p^*}(\Omega)$. As a borderline case, when $p = n$, Pohozaev \cite{Po}, Trudinger \cite{Tru} and Yudovic \cite{Yu} independently discovered the embedding of exponential type on bounded domains $\Omega\subset \mathbb{R}^n$.  Namely, the embedding $W_{0}^{1,n}\left(\Omega\right)  \subset L_{\varphi_{n}}\left(  \Omega\right)  $,  where  $L_{\varphi_{n}}\left( \Omega\right)  $ is the
Orlicz space associated with the Young function $\varphi_{n}(t)=\exp\left(\beta\left\vert t\right\vert ^{n/(n-1)}\right)  -1$ for some $\beta>0$. Subsequently,  using the Schwarz rearrangement, Moser \cite{Mo} sharpened this embedding and proved the following:

\begin{thm}\label{Trudinger-Moser}
  There exists a constant $C$ such that the following inequality holds

  \begin{equation}\label{Trudinger-Moser Inequality}
    \sup_{u \in W_0^{1,n}(\Omega), ||\nabla u||_{L^n(\Omega)} \leq 1} \frac{1}{ |\Omega| } \int_\Omega e^{\alpha_n |u|^{\frac{n}{n-1}}} dx \leq C
  \end{equation}
  where $\alpha_n = n \omega_{n-1}^{\frac{1}{n-1}}$ and $\omega_{n-1}$ denotes the Hausdorff measure of $\mathbb{S}^{n-1}$. Moreover, (\ref{Trudinger-Moser Inequality}) is sharp in the sense that for any $\beta > \alpha_n$, there exists a sequence $\{ u_k \}$ satisfying $u_k \in W_0^{1,n}(\Omega)$ and $||\nabla u_k||_{L^n (\Omega)} \leq 1$, such that the integration in (\ref{Trudinger-Moser Inequality}) blows up, i.e. $\int_\Omega e^{\beta |u_k|^{\frac{n}{n-1}}} dx \to \infty$ as $k \to \infty$.
\end{thm}

Similar to (\ref{Trudinger-Moser Inequality}), Moser in \cite{Mo} also proved an analogous inequality with the same best constant on 2-sphere. More precisely, he showed

\begin{thm}\label{Trudinger-Moser On 2-sphere} Let $d\mathcal{H}_2$ be the surface measure on two-dimensional sphere $\mathbb{S}^2$.
Then there exists a uniform constant $C>0$ such that
for any $u \in W^{1,2}(\mathbb{S}^2)$ with $||\nabla u||_{L^2(\mathbb{S}^2)} \leq 1$ and $\int_{\mathbb{S}^2} u d\mathcal{H}_2 = 0$,
the following holds:
  \begin{equation}\label{Trudinger-Moser Inequality On 2-sphere}
    \int_{\mathbb{S}^2} e^{4 \pi u^2} d \mathcal{H}_2\leq C,
  \end{equation}
  Moreover, (\ref{Trudinger-Moser Inequality On 2-sphere}) is sharp in the same sense that the above inequality fails to hold uniformly.
\end{thm}

In addition, (\ref{Trudinger-Moser Inequality On 2-sphere}) can be improved  if one only considers symmetric functions satisfying $u(\xi)=u(-\xi)$ for all $\xi\in \mathbb{S}^2$.

\begin{thm}\label{Improved Trudinger-Moser On 2-sphere}
  There exists a uniform constant $C$ such that

  \begin{equation*}
    \int_{\mathbb{S}^2} e^{8 \pi |u|^2} d \mathcal{H}_2< C
  \end{equation*}
  holds for those $u$ satisfying the same assumption as in Theorem \ref{Trudinger-Moser On 2-sphere}, and $u(\xi) = u(-\xi)$.

\end{thm}

\medskip

A very interesting problem in complex analysis related to Theorem \ref{Trudinger-Moser}, Theorem \ref{Trudinger-Moser On 2-sphere} and Theorem \ref{Improved Trudinger-Moser On 2-sphere} was proposed in \cite{Mo}. To illustrate this problem, we let $f(z)$  be any analytic function defined in complex unit disc $\{ z\in \mathbb{C}: \ |z| < 1 \}$ and satisfy
\begin{equation}\label{Assumption}
  \int \int_{|z| < 1} |f'(z)|^2 dx dy \leq \pi ; \ \ \ \ f(0) = 0.
\end{equation}

The question is: does there exist uniform constants $\alpha > 0$ and $C$ such that the following holds
\begin{equation}\label{Beurling's Question}
  \int_{\mathbb{S}^1} e^{\alpha |f(z)|^2} dz \leq C
\end{equation}
for all $f$ satisfying (\ref{Assumption})?

It was first pointed out by A. Beurling \cite{Be} that for such a function satisfying (\ref{Assumption}), we have an upper bound for its level set, i.e.

\begin{equation}\label{Beurling's Estimate}
  \| \{ \theta: |f(e^{i \theta})| \geq s \} \| \leq e^{-s^2 + 1}
\end{equation}

This estimate instantly implies (\ref{Beurling's Question}) holds for $\alpha < 1$ by rewriting the integral in (\ref{Beurling's Question}) in terms of the distributional function. Moreover, he proved that the following so-called Beurling's function

\begin{equation}\label{Beurling's Function}
  B_a (e^{i \theta}) = \( \log \frac{1}{1-a e^{i \theta}} \) / \sqrt{\log \frac{1}{1 - a^2}}
\end{equation}
fails (\ref{Beurling's Question}) when $\alpha > 1$.

\medskip
The critical case $\alpha=1$ was a very difficult  open question until
  S.-Y. A. Chang and D. E. Marshall \cite{CM} proved the following celebrated theorem.

\begin{thm}\label{Chang-Marshall}
  For those $f(z)$ satisfying assumption (\ref{Assumption}), there exists a constant $C > 0$ such that
  \begin{equation}\label{Chang-Marshall Inequality}
    \sup_{|z| < 1} \int_0^{2 \pi} e^{ \alpha |f(e^{i \theta}) - f(z)|^2} P_z (\theta) \frac{d \theta}{\pi} \leq C
  \end{equation}
  where $P_z(\theta)$ is the Poisson kernel, $0 \leq \alpha \leq 1$. Moreover, (\ref{Chang-Marshall Inequality}) is sharp in the sense that for $\alpha > 1$, the integral will be made arbirtarily large with Beurling's function \eqref{Beurling's Function}.
\end{thm}

The following corollary is a real-variable version of Theorem \ref{Chang-Marshall}. For $f(z)$ defined in the unit disc $\{ z\in \mathbb{R}^2 : \ |z| \leq 1 \}$,$f(0) = 0$ and belongs to $L^1(\mathbb{S}^1)$, we have

\begin{cor}\label{Real Chang-Marshall}
  There exists $C > 0$ such that if $f$ is real-valued with $\int_0^{2 \pi} f(e^{i \theta}) d \theta = 0$ and $||\nabla f||_{L^2} \leq \pi^{1/2}$, then

  \begin{equation}\label{Real Chang-Marshall Inequality}
    \sup_{|z|< 1} \int_0^{2\pi} e^{\alpha |f(e^{i \theta}) - f(z)|^2} P_z(\theta) \frac{d\theta}{2 \pi} \leq C
  \end{equation}

  holds for $0 \leq \alpha \leq 1$. Moreover, (\ref{Real Chang-Marshall Inequality}) is sharp in the sense that the integral in (\ref{Real Chang-Marshall Inequality}) will be made arbitrarily large for $\alpha > 1$ with the function $Re B_a(e^{i \theta})$.
\end{cor}

\medskip

Chang-Marshall's result is only known for unit disc in dimension two. S. Y. Chang raised the following question to G. Lu in 2015: can we get something similar for Sobolev functions in higher dimensional balls or even more general domains in $\mathbb{R}^n$ for $n\geq 3$?
This is the main purpose of this article.
 Notice that the Chang-Marshall inequality in Theorem \ref{Chang-Marshall} and Corollary \ref{Real Chang-Marshall} can be viewed as a Trudinger-Moser type trace inequality for Sobolev functions (harmonic functions)  with mean value zero on disks in dimension 2. Compared with Theorem \ref{Trudinger-Moser On 2-sphere} for Moser's inequality on two dimensional sphere $\mathbb{S}^2$, what we are looking for is actually a more general Trudinger-Moser trace inequality for Sobolev functions.

\medskip

\medskip

The main purpose of the present paper is to establish a Trudinger-Moser type trace inequality as well as giving a upper bound of the best constant. Roughly speaking, we want to prove that for any $u \in C^1(\overline{\Omega}) \cap C(\p \Omega)$ with $||\nabla u||_{L^n(\Omega)} \leq 1$ and $\int_{\Omega} u dx= 0$, where $\Omega$ is a bounded domain with $C^{1}$ boundary, there exist constants $\alpha, C$ such that

\begin{equation}\label{Trace inequality}
  \int_{\p \Omega} e^{\alpha |u|^{\frac{n}{n-1}}} d \mathcal{H}_{n-1} \leq C.
\end{equation}

Our initial attempt was to try to establish the above trace inequality by following Adams' approach \cite{Ad}.
In other words, if we can represent the function $u$ on the boundary $\partial \Omega$ in terms of some integrals of the functions $u$ and $\nabla u$ on $\Omega$, then we may carry out the argument of Adams using O'Neil's lemma to accomplish the proof of such an inequality.
  However, the integral (\ref{Trace inequality}) and the Dirichlet norm are taken on metric spaces with different Hausdorff dimension, the classical method of Adams does not seem  to be applicable here. To overcome this difficulty, we develop a new approach in the present paper. The idea is that we are going to look for a so-called trace function, say $w$, satisfying certain properties so that we can convert the integral over the boundary $\partial \Omega$ of the domain $\Omega$ in (\ref{Trace inequality}) to an integral in the interior of $\Omega$.  Our proof is inspired by the work on $L^p$ Sobolev trace inequalities of G. Auchmuty, see \cite{Au}. On the other hand, we get an upper bound of $\alpha$, by constructing a counter sequence $\{ u_k \}$. This sequence behaves quite similar as the one in the proof of the sharpness of
  (\ref{Trudinger-Moser Inequality}). See \cite{Mo} for more details of such kind of sequences.

Our paper is organized as following. In Section 2, we will state our main results. In Section 3, we will give the proof of the existence of the trace function. In Section 4, we prove the Chang-Marshall inequality and we will give the proof of the upper bound of best constant in Section 5.

\medskip

{\bf Acknowledgement.} The authors wish to thank  S. Y. Alice Chang for suggesting this problem to us and for her encouragement, and for her comments of its exposition on the earlier draft of the paper. The second author also wish to thank G. Auchmuty for bringing his work \cite{Au}
to our attention.

\section{Statement of Main Results}

Our first result is a generalized version of Theorem \ref{Chang-Marshall} for arbitary n-dimensional bounded domain $\Omega$ with $C^1$ boundary for all $n\geq 2$.

\begin{thm}\label{Main Result 1}
  In a bounded open domain $\Omega \subset \mathbb{R}^n$ with $C^{1,\alpha}$ boundary, for any $\alpha < \beta_n = (n-1)(\omega_{n-1}/2)^{\frac{1}{n-1}}$ and for any $u \in C^1(\overline{\Omega}) \cap C( \partial \Omega)$ with $\int_{\Omega} u dx = 0$ and $||\nabla u||_{L^n(\Omega)} \leq 1$, there exists a constant $C = C(n,\alpha,\Omega)$ such that

  \begin{equation}\label{Trudinger-Moser Trace Inequality}
    \int_{\p \Omega} e^{\alpha |u|^{\frac{n}{n-1}}} d\mathcal{H}_{n-1} \leq C.
  \end{equation}

\end{thm}

Our second result shows the sharpness of the constant $\beta_n = (n-1)(\omega_{n-1}/2)^{\frac{1}{n-1}}$. More precisely,

\begin{thm}\label{Main Result 2}
  Inequality (\ref{Trudinger-Moser Trace Inequality}) is sharp in the sense that for any $\alpha > \beta_n$, there exists a sequence $\{ u_k \} \subset C^1(\overline{\Omega}) \cap C(\p \Omega)$ satisfying $\int_{ \Omega} u_k = 0$ and $||\nabla u_k||_{L^n} \leq 1$ such that

  \begin{equation}\label{Integral Blows Up}
    \int_{\p \Omega} e^{\alpha |u_k|^{\frac{n}{n-1}}} d\mathcal{H}_{n-1} \to \infty \ \ \ \ as \ k \to \infty.
  \end{equation}
\end{thm}

\medskip

\section{Existence of  the Trace Function}

Let $\Omega$ be a bounded domain of $\mathbb{R}^n$ with $C^{1,\alpha}$ boundary. We are trying to look for a solution of the torsion equation with Neumann boundary condition. More precisely, we consider the following functional

\begin{equation}\label{Functional}
  E(u) = \int_{\Omega} |\nabla u|^p dx- p \int_{\p \Omega} u d\mathcal{H}_{n-1},
\end{equation}
where $u \in W^{1,p}(\Omega) $ with $1 < p < n$ and $ \int_\Omega u dx = 0$. For simplicity, we denote such a function $u$ as $u \in W_m^{1,p}(\Omega)$. We have the following properties of such functional.

\begin{lemma}\label{Continuity}
  $E : W^{1,p}_m \to \mathbb{R}$ is continuous with respect to $|| \cdot ||_{W^{1,p}}$.
\end{lemma}

\begin{proof}
The   continuity directly follows from the following estimate:

  \begin{eqnarray*}
    E(u) &=& \int_{\Omega} |\nabla u|^p dx- p \int_{\p \Omega} u d\mathcal{H}_{n-1} \\
         &\leq& ||u||_{W^{1,p}}^p + p |\Omega |^{\frac{p-1}{p}} ||u||_{L^p(\p \Omega)} \\
         &\leq & ||u||_{W^{1,p}}^p + C ||u||_{W^{1,p}},
  \end{eqnarray*}
  where we apply the classic trace inequality for Sobolev functions on the last line.
\end{proof}

\begin{lemma}\label{Coercivity}
  The functional $E : W^{1,p}_m \to \mathbb{R}^n$ is coercive, i.e. $E(u) \to \infty$ as $||u||_{W^{1,p}} \to \infty$
\end{lemma}

\begin{proof}
  For the first term in (\ref{Functional}), we have

  \begin{equation*}
    ||\nabla u||_{L^p}^p \gtrsim ||u||_{W^{1,p}}^p.
  \end{equation*}

  This follows from Poincar\'e's inequality. For the second term in (\ref{Functional}), we have from the proof of Lemma \ref{Continuity} that

  \begin{equation*}
   | \int_{\p \Omega} u | \lesssim ||u||_{W^{1,p}}.
  \end{equation*}

  Therefore, we obtain a lower bound estimate of the functional

  \begin{equation}\label{Lower bound estimate}
    E(u) \geq C_1 ||u||_{W^{1,p}}^p - C_2 ||u||_{W^{1,p}},
  \end{equation}
  since $p > 1$, this leads to the conclusion.
\end{proof}

\begin{lemma}\label{Convexity}
  $E : W^{1,p}_m \to \mathbb{R}^n$ is convex.
\end{lemma}

\begin{proof}
  It is easy to verify that $\{ \frac{d^2}{dt^2} E(u+ t\phi) \} \vert_{t=0} > 0$ for any $u,\phi \in W_m^{1,p}(\Omega)$.
\end{proof}

Combining Lemmas \ref{Continuity},\ref{Coercivity} and \ref{Convexity}, we conclude that $E : W^{1,p}_m \to \mathbb{R}^n$ has a unique minimizer of $E$ in $W_m^{1,p}(\Omega)$. For simplicity, we denote this function as $w$, which has many good properties that will be crucial for our proof of the main results. For the first variation,

\begin{equation*}
  0= \{ \frac{d}{dt} E(w + t\phi) \} \vert_{t=0} = p \int_\Omega |\nabla w|^{p-2} \nabla w \nabla \phi - p \int_{\p \Omega} \phi.
\end{equation*}

Without the mean value zero constraint on $\phi$, we have the following

\begin{equation}\label{Constraint free}
  \int_\Omega |\nabla w|^{p-2} \nabla w \nabla \phi - \int_{\p \Omega} \phi = - \frac{|\p \Omega |}{|\Omega|} \int_{\Omega} \phi
\end{equation}
holds for any $\phi \in W^{1,p}(\Omega)$. In particular, for $\phi \in W_0^{1,p}(\Omega)$, after taking integration by part, we get

\begin{equation}\label{Interior property}
  \triangle_p w = div(|\nabla w|^{p-2}\nabla w) = \frac{|\p \Omega |}{|\Omega|}.
\end{equation}

Thus from (\ref{Constraint free}), for any $\phi \in W^{1,p}(\Omega)$, we have

\begin{equation*}
  -\int_\Omega div(|\nabla w|^{p-2} \nabla w) \phi + \int_{\p \Omega} \phi |\nabla w|^{p-2}\nabla w \cdot \overrightarrow{n} - \int_{\p \Omega} \phi = -\frac{|\p \Omega |}{|\Omega|} \int_{\Omega} \phi.
\end{equation*}

This means

\begin{equation}\label{Boundary Property}
  \int_{\p \Omega} \phi \( ( |\nabla w|^{p-2} \nabla w ) \cdot \overrightarrow{n} - 1 \) = 0.
\end{equation}

\medskip

\section{Proof of Theorem \ref{Main Result 1}}

From the boundary property (\ref{Boundary Property}) of $w$, we first convert the boundary integral in (\ref{Main Result 1}) to integral on $\Omega$. Indeed, we have the following

  \begin{eqnarray}
    \int_{\p \Omega} e^{\alpha |u|^{\frac{n}{n-1}}} d\mathcal{H}_{n-1}
      &=& \int_{\p \Omega} e^{\alpha |u|^{\frac{n}{n-1}}} | \nabla w |^{p-2}\nabla w \cdot \overrightarrow{n} d\mathcal{H}_{n-1}\nonumber \\
      &=& \int_{\Omega} div\( e^{\alpha |u|^{\frac{n}{n-1}}} | \nabla w |^{p-2}\nabla w \) dx\nonumber\\
      &=& \int_\Omega e^{\alpha |u|^{\frac{n}{n-1}}} \frac{\alpha n}{n-1} |u|^{\frac{1}{n-1}} \nabla u (| \nabla w |^{p-2}\nabla w) dx \nonumber\\
      &+& \int_\Omega e^{\alpha |u|^{\frac{n}{n-1}}} \triangle_p w dx  \label{Interior Integral}.
  \end{eqnarray}

Since $\triangle_p w = \frac{|\p \Omega |}{|\Omega|}$, from Theorem 1.2 of \cite{C} , we know that the second term in (\ref{Interior Integral}) is uniformly bounded provided that $\alpha< \beta_n < n \left( \frac{\omega_{n-1}}{2} \right)^{\frac{1}{n-1}}$. For the first term in (\ref{Interior Integral}), since $p < n$, from H\"older's inequality,

  \begin{eqnarray*}\label{Estimate for 1st term}
    &&\int_\Omega e^{\alpha |u|^{\frac{n}{n-1}}} \frac{\alpha n}{n-1} |u|^{\frac{1}{n-1}} \nabla u (| \nabla w |^{p-2}\nabla w) dx \\
      &\leq & (\int_\Omega e^{p^* \alpha |u|^{\frac{n}{n-1}}} |u|^{\frac{p^*}{n-1}} dx )^{1/p^*} ||\nabla u||_{L^n} ||\nabla w||_{L^p}^{p-1} \\
      &\lesssim & (\int_\Omega e^{p^* \alpha |u|^{\frac{n}{n-1}}} |u|^{\frac{p^*}{n-1}} dx )^{1/p^*} \\
      &\lesssim & \( \int_\Omega e^{ (\alpha p^* +\epsilon) |u|^{\frac{n}{n-1}} } dx \)^{1/p^*},
  \end{eqnarray*}
where $p^*$ is the Sobolev conjugate of $p$, i.e. $\frac{1}{p^*} = \frac{1}{p} - \frac{1}{n}$, $\epsilon$ is a fixed small enough number. Following a similar argument as we did in the estimate of the second term of (\ref{Interior Integral}), we find that once $(\alpha p^* +\epsilon) \leq n \left( \frac{\omega_{n-1}}{2} \right)^{\frac{1}{n-1}}$,  from \cite{C}, the first term of (\ref{Interior Integral}) is bounded. From the above argument, for any $\alpha < \beta_n = (n - 1) (\omega_{n-1}/2)^{\frac{1}{n-1}}$, we can pick $p$ close enough to $1$ (hence $p^*$ close enough to $\frac{n}{n-1}$), $\epsilon$ small enough such that

\begin{equation*}
  (\alpha p^* +\epsilon) \leq n \left( \frac{\omega_{n-1}}{2} \right)^{\frac{1}{n-1}},
\end{equation*}
which implies both the first term and second term in (\ref{Interior Integral}) are bounded and thus we complete our proof.

\medskip

\section{Proof of Theorem \ref{Main Result 2}}

For any two points $x,y \in \overline{\Omega}$, we denote $|x-y|$ as the usual distance in Euclidean space. In specific, when $x,y \in \p \Omega$, we denote $d(x,y)$ as the geodesic distance along the boundary surface between this two points. Since $\p \Omega$ is of class $C^{1,\alpha}$, in a small neighborhood of $y \in \p \Omega$, say $B_\delta (y)$, which is n-dimensional ball centered at y with radius $\delta$, we have for any $x \in \p \Omega$, $|x - y| \leq d (x,y) \leq (1+O(\delta)) |x-y|$.

\medskip

Fix $y\in \p \Omega$, we can directly assume $\delta = 1$ and everything works for arbitrary $\delta$ after a proper modification. Also, without loss of generality, we can assume that $B_1(y)$ are disjoint and the size of $\Omega$ is much larger than $B_1(y)$, i.e. $\frac{|B_1 (y)|}{|\Omega|} < \epsilon$, where $\epsilon$ is left to be chosen. We can then define our counter sequence $\{ u_r(x) \}$ in each ball as follows:

\medskip
For $x \in B_1(y) \cap \Omega$,

\begin{equation*}\label{Counter Sequence}
  u_r(x) = \begin{cases}
             1 \ \ \ &0 \leq |x-y| \leq r \\
             \log \frac{1}{|x-y|}/ \log \frac{1}{r} \ \ \ &r \leq |x-y| \leq 1 \\
             0 \ \ \ &1 \leq |x-y|,
           \end{cases}
\end{equation*}

and we further let $u_r$ vanish everywhere else. Due to the $C^{1,\alpha}$ boundary condition, $| \frac{1}{|\Omega|} \int_{\Omega} u_r dx| < o(\epsilon)$ and following an easy calculation, the Dirichlet norm of $u_r$ is

\begin{eqnarray*}
  ||\nabla u_r||_{L^n}^n &=& \int_{B_1(y) \cap \Omega} |\nabla u_r|^n dx  \\
                         &=& \frac{1}{2} (\log \frac{1}{r})^{-(n-1)} \omega_{n-1} ( 1 + o(\epsilon) ).
\end{eqnarray*}

Now we test the function $(u_r - \frac{1}{| \Omega|}\int_{\Omega} u_r dx)/||\nabla u_r||_{L^n}$ into the integral of (\ref{Main Result 1}), for any $\alpha > \beta_n$, we have the following estimate

\begin{eqnarray*}\label{Estimate of The Counter Sequence}
    && \int_{\p \Omega} e^{\alpha |(u_r - \frac{1}{|\Omega|}\int_{ \Omega} u_r dx)/||\nabla u_r||_{L^n}|^{\frac{n}{n-1}}} d\mathcal{H}_{n-1} \\
    &\geq & \int_{B_r(y) \cap \p \Omega} e^{\frac{\alpha (1 - o(\epsilon))}{||\nabla u_r||_{L^n}^{\frac{n}{n-1}}}} d\mathcal{H}_{n-1} \\
    &\geq & |B_r(y) \cap \p \Omega | e^{\alpha (1 - o(\epsilon)) \log(1/r) (\omega_{n-1}/2)^{-\frac{1}{n-1}}} \\
    &\gtrsim & r^{n-1} \( \frac{1}{r} \)^{\alpha (1 - o(\epsilon)) (\omega_{n-1}/2)^{-\frac{1}{n-1}}} \\
    &\to & \infty  \ \ \ as \ r \to 0.
\end{eqnarray*}

The last line is true once we pick $\epsilon$ small enough such that $\alpha (1 - o(\epsilon)) > \beta_n = (n-1)(\omega_{n-1}/2)^{\frac{1}{n-1}}$.


\begin{thebibliography}{99}

\bibitem {Ad}D. Adams, A sharp inequality of J. Moser for higher order
derivatives. Ann. of Math. (2) 128 (1988), no. 2, 385--398.

\bibitem {Au} G. Auchmuty. Sharp boundary trace inequalities. Proc. Roy. Soc. Edinburgh Sect. A 144 (2014), no. 1, 1–12.

\bibitem {Be}A. Beurling, Études sur un problème de majoration. Imprimerie Almquist and Wiksell, 1933.

\bibitem {C} A. Cianchi, Moser-Trudinger inequalities without boundary conditions and isoperimetric problems. Indiana Univ. Math. J. 54 (2005), no. 3, 669-705.
 
  
\bibitem {CM} S-Y. A. Chang and D. E. Marshall. On a sharp inequality concerning the Dirichlet integral.
Amer. J. Math. 107 (1985), no. 5, 1015–1033.



 



\bibitem {Mo}J. Moser,  A sharp form of an inequality by N. Trudinger.
Indiana Univ. Math. J. 20 (1970/71), 1077–1092.

\bibitem{ON}R. O'Neil, Convolution operators and L(p,q) spaces. Duke Math. J. 30 (1963), 129–142.

\bibitem {Po}S.I. Poho\v{z}aev, \textit{On the Sobolev embedding in the case
}$pl = n$, Proceedings of the Technical Scientific
Conference on Advances of Scientific Research 1964-1965. Mathematics Section, 158-170, Moskov.
Energet. Inst., Moscow, 1965

 

\bibitem {Tru}N. Trudinger, On imbeddings into Orlicz spaces and some applications.
J. Math. Mech. 17 1967 473–483.

\bibitem {Yu}V.I. Yudovic, Some estimates connected with integral
operators and with solutions of elliptic equations. (Russian) Dokl. Akad. Nauk
SSSR 138 1961 805--808.

\end{thebibliography}
\end{document}